\documentclass[a4paper, 12pt,oneside,reqno]{amsart}
\usepackage[a4paper]{geometry}
\usepackage[english]{babel}
\usepackage[utf8x]{inputenc}
\usepackage[T1]{fontenc}
\usepackage{amsfonts}
\usepackage{mathrsfs}

\usepackage[sc]{mathpazo}
\usepackage{tikz}
\usetikzlibrary{arrows,shapes,snakes,automata,backgrounds,petri,through,positioning}
\usetikzlibrary{intersections}

\usepackage{tikz-cd}

\usepackage[matrix,arrow]{xy}
\usepackage{amssymb,amscd,amsthm,amsmath}
\usepackage{mathtools}

\usepackage[a4paper]{geometry}

\usepackage{amsmath}
\usepackage{amssymb}
\usepackage{amsthm}
\usepackage{graphicx}
\usepackage[colorinlistoftodos]{todonotes}
\usepackage[colorlinks=true, allcolors=blue]{hyperref}

\usepackage{BOONDOX-frak} 

\numberwithin{equation}{section}

\newtheorem{theorem}{Theorem}[section]
\newtheorem{lemma}[theorem]{Lemma}
\newtheorem{proposition}[theorem]{Proposition}
\newtheorem{corollary}[theorem]{Corollary}

\theoremstyle{definition}
\newtheorem{definition}[theorem]{Definition}
\newtheorem{example}[theorem]{Example}
\newtheorem{remark}[theorem]{Remark}

\usepackage{array}
\newcolumntype{P}[1]{>{\centering\arraybackslash}p{#1}}
\usepackage{multirow}

\title{Describing realizable Gauss diagrams using the concepts of parity or bipartate graphs}

\author{Alexei Lisitsa}

\address{Department of Computer Science, University of Liverpool, Liverpool, UK}

\author{Viktor Lopatkin}

\address{National Research University Higher School of Economics, Faculty of Computer Science, Pokrovsky Boulevard 11, Moscow, 109028 Russia;\ wickktor@gmail.com}

\author{Alexei Vernitski}

\address{Department of Mathematical Sciences, University of Essex, Colchester, UK}

\begin{document}

\maketitle
\begin{abstract}
Two recent publications describe realizable Gauss diagrams using conditions stating that the number of chords in certain sets of chords is even or odd. We demonstrate that these descriptions are incorrect by finding multiple counter-examples. However, the idea of having a parity-based description of realizable Gauss diagrams is attractive. We recall that realizability of Gauss diagrams as touch curves can be described via bipartite graphs. We show that realizable Gauss diagrams can be described via bipartite graphs.
\end{abstract}

\section{Introduction}

A Gauss diagram captures some information about a closed plane curve. This is not a one-to-one correspondence; that is, two different closed plane curves can have the same Gauss diagram. An interlacement graph captures some information about a Gauss diagram. This is not a one-to-one correspondence; that is, two different Gauss diagram can have the same interlacement graph. 

Not every Gauss diagram corresponds to a closed plane curve. If a Gauss diagram does correspond to a closed plane curve, the Gauss diagram is called realizable. A remarkable fact which has hardly ever been stated explicitly is that it is possible to decide whether a Gauss diagram is realizable just by inspecting its interlacement graph; that is, the interlacement graph retains enough information about the Gauss diagram to decide if the Gauss diagram is realizable. 

As we describe below, several recent publications \cite{GL18,B19,GL} offer especially simple elegant descriptions of realizability of Gauss diagrams expressed (implicitly) in the language of their interlacement graphs. Every condition in these descriptions is based on parity, that is, it checks whether a certain set of edges in the graph has an even or odd size. Unfortunately, as we show below, these descriptions are wrong, and we generate a family of counterexamples. 

Nevertheless, the idea of having a parity-based description of realizable Gauss diagrams in the language of their interlacement graphs is attractive, so in this paper we achieve one possible description of this kind. On the one hand, we note that realizability of Gauss diagrams as touch curves can be described via bipartite graphs. On the other hand, we note that a recent publication \cite{STZ09} gives a (very inefficient) description of realizability of Gauss diagrams expressed (implicitly) in the language of their interlacement graphs. Combining the two ideas, we show that a Gauss diagram is realizable if and only if a certain modification of its interlacement graph is bipartite. Not only this description is theoretically interesting as based on parity and on the interlacement graphs of Gauss diagrams, but also it can be easily implemented as a fast and simple algorithm for checking realizability of Gauss diagrams.

\section{An overview of the history of realizability}

The Gauss code of the curve was first defined by C.F. Gauss \cite{Gauss}. This sequence also determines a chord diagram which in the context of the study of plane curves is called a Gauss diagram, and the question raised by Gauss is \textit{to recognize which chord diagrams arise from some plane curve.}

The Gauss problem about realizability of a chord diagram is an old one and has been solved in many different ways. We briefly survey some known realizability descriptions for Gauss diagrams which are most relevant to our study. In 1936, M. Dehn \cite{D} solved the Gauss problems for the first time, finding an algorithmic solution based on the existence of a touch Jordan curve which is the image of a transformation of the knot diagram by successive splits replacing all the crossings. (An improved and clearer version of Dehn's algorithm can be found in \cite{RT84}.)
Much later,  in 1976, L. Lovasz and M.L. Marx \cite{LM} found  alternative criteria for realizability. 
At the same time, R.C. Read and P. Rosenstiehl \cite{R76,RR} showed that one can express realizability in terms of \emph{interlacement graphs}.  
The last characterization is based on the tripartition of such graphs into cycles, cocycles and bicycles. H. de Fraysseix and P. Ossona de Mendz in \cite{dFOdM97}, using a modification of Dehn's ideas, obtained a new characterization of Gauss codes which led them to a short self-contained proof of Rosenstiehl's characterization. Next, B. Shtylla, L. Traldi and L. Zulli in \cite{STZ09} reformulated answers to the Gauss problem given by P. Rosenstiehl and by H. de Fraysseix and P. Ossona de Mendz to give new graph-theoretic and algebraic characterizations of realizable Gauss codes. 

We also refer the reader to the book of \'E. Ghys \cite[Gauss is back: curves in the plane]{Ghys} for more discussion and bibliography, and other realization criteria. See also \cite{M} for another parity-based approach to chord diagrams.

According to M. Dehn result \cite{D, RT84} the Gauss problem reduces to a touch realization of a diagram obtained from the first one by some transformation. Further, being touch-realizable is equivalent to the corresponding interlacement graph being bipartite. 



\section{Definitions}



Consider a plane curve $\gamma$ whose self-intersections are double points, as in Figure Fig.\ref{fig:gauss-ex} a. Assume that each crossing is labelled by some label. Walk along the curve until returning back to where we started, and generate a word $w$ which records the labels of the crossings in the order we pass over them on the curve. For example, starting in Figure Fig.\ref{fig:gauss-ex} a at the crossing 1 and moving along the curve to the right initially (and turning as the curve turns), we obtain $w=12334124$. The word $w$ is a double occurrence word, that is, each letter features in it exactly twice. It is called the \textit{Gauss code} of the curve $\gamma$. Obviously, a curve can have many Gauss codes, since the definition of the Gauss code depends on the starting point and the direction of the walk around the curve. To capture all Gauss codes of a curve in one object, it is convenient to represent Gauss codes of a curve diagrammatically; namely, if we place the letters of a Gauss code $w$ around a circle in the order in which the letters occur in $w$ and if we join up each pair of identical letters by a chord, then the obtained chord diagram is called the Gauss diagram $\mathfrak{G}(\gamma)$ of the plane curve (see Fig.\ref{fig:gauss-ex} a), and b)).

\begin{figure}[h!]
    \centering
    \begin{tikzpicture}[scale=0.7]
  \draw[line width =2, name path= a](0,0) to [out= 30, in = 30] (-0.5, 2.5);
\draw[line width =2, name path= b](-0.5,2.5) to [out= 210, in = 180] (1.2, -1);
\draw[line width =2, name path= c] (1.2,-1) to [out= 0, in = 0] (-0.5, 3.3);
\draw[line width =2, name path= d] (-0.5,3.3) to [out= 180, in = 160] (-2, -1);
\draw[line width =2, name path= e] (-2,-1) to [out= 340, in = 270] (1.2, 0.5);
\draw[line width =2, name path= f] (1.2,0.5) to [out= 90, in = 0] (-1.2, 1.5);
\draw[line width =2, name path= g] (-1.2,1.5) to [out = 180, in = 90] (-2, 0.75) to [out= 270, in = 210] (0, 0);

 \fill [name intersections={of=f and b, by={1}}]
(1) circle (3pt) node[above left] {$1$};

 \fill [name intersections={of=a and f, by={2}}]
(2) circle (3pt) node[above right] {$2$};

 \fill [name intersections={of=e and b, by={3}}]
(3) circle (3pt) node[above] {$3$};

 \fill [name intersections={of=g and b, by={4}}]
(4) circle (3pt) node[below] {$4$};
\begin{scope}[xshift = 6cm, yshift = 1cm]
     \draw[line width =2] (0,0) circle (2.2);
     
     {\foreach \angle/ \label in
   { 90/1, 135/4, 180/2, 225/1, 270/4, 315/3, 0/3, 45/2}
   {
    \fill(\angle:2.5) node{$\label$};
    \fill(\angle:2.2) circle (3pt) ;
    }
}
     \draw[line width = 2] (90:2.2) -- (225:2.2);
     \draw[line width = 2] (135:2.2) -- (270:2.2);
     \draw[line width = 2] (180:2.2) -- (45:2.2);
     \draw[line width = 2] (0:2.2) to [out = 180, in = 140] (315:2.2);
     
\end{scope}
\begin{scope}[xshift =12cm, yshift = 0.5cm]
  \draw[line width = 2] (90:1.7) -- (200:1.7);
  \draw[line width = 2] (90:1.7) -- (340:1.7);
  \draw[line width = 2] (200:1.7) -- (340:1.7);
  
  \fill(90:2) node{$1$};
  \fill(90:1.7) circle(3pt);
  
  \fill(200:2) node{$2$};
  \fill(200:1.7) circle(3pt);
  
  \fill(340:2) node{$4$};
  \fill(340:1.7) circle(3pt);
  
  \fill(40:3) node{$3$};
  \fill(40:2.7) circle(3pt);
  
 \end{scope}
  \fill(0,-2.5) node{$a)$};
  \fill(5,-2.5) node{$b)$};
  \fill(10,-2.5) node{$c)$};
    \end{tikzpicture}
    \caption{Example of a) a planar curve; b) its Gauss diagram and c) its interlacement graph. The corresponding Gauss code is $\mathbf{12334124}$}
    \label{fig:gauss-ex}
\end{figure}
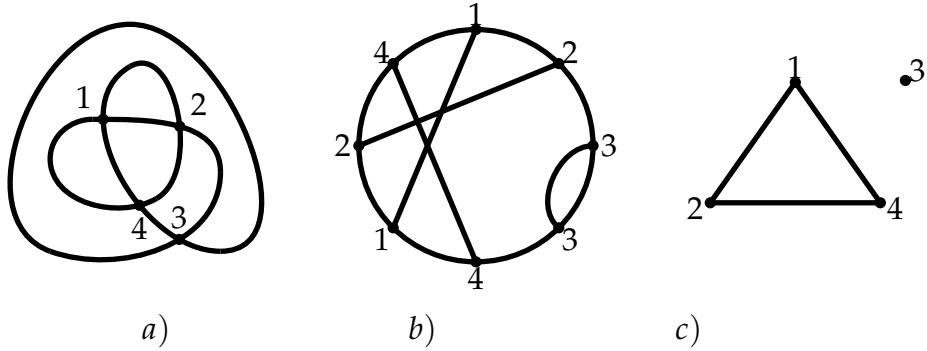

A Gauss diagram $\mathfrak{G}$ is called \textit{realizable} if there is a closed plane curve $\gamma$ such that $\mathfrak{G} = \mathfrak{G}(\gamma)$.

With any chord diagram one can associate a graph, which is called a a \emph{circle graph} or a \textit{chord-intersection graph} or an \textit{interlacement graph}, whose set of vertices is the set of all chords of the chord diagram, and in which a pair of two vertices $u, v$ is connected with an edge if and only if the chords $u, v$ intersect (see Fig. \ref{fig:gauss-ex} c).

\section{Realizability descriptions based on interlacement graphs} 

The first of the following two conditions was known to Gauss, and the second one features in a number of recent papers, starting from \cite{CE93,CE96}. 

\begin{definition}[{The Evenness Conditions}]\label{evennes}
 We say that a Gauss diagram $\mathfrak{G}$ satisfies the evenness conditions if its interlacement graph $\Gamma(\mathfrak{G})$ has the
following properties:
\begin{enumerate}
    \item the degree of each vertex is even,
    \item each pair of non-neighboring vertices has an even number of common neighbors (possibly, zero).
\end{enumerate}
\end{definition}

It is well known that the evenness conditions are necessary but not sufficient for a Gauss diagram to be realizable; see \cite{CE96}, and counterexamples below in Figures \ref{counterexample_to_n=9}, \ref{fig:counter10}, \ref{fig_for_STZ1_not_realizable}.

For the first time the relizability description for Gauss words and diagrams expressed solely in terms of interlacement graphs was found by P. Rosenstiehl in \cite{R76}. Following  \cite{dFOdM97} Rozentiehl's conditions can be formulated as follows. 

\begin{theorem}[The Rosentiehl Criteria of Realizability, {\cite{R76, dFOdM97}}]~\\
 A Gauss diagram $\mathfrak{G}$ is realizable if and only if its interlacement graph $\Gamma(\mathfrak{G}) = (V,E)$  has the following properties:
 \begin{enumerate}
     \item the degree of each vertex is even,
     \item there is a subset of vertices $A \subset V$ of $\Gamma(\mathfrak{G})$ such that the following two conditions are equivalent for any two vertices $u$ and $v$:  
    \begin{enumerate}
    \item[i)] the vertices $u$ and $v$ have an odd number of common neighbors, 
    \item[ii)] the vertices $u$ and $v$ are neighbors and either both are in $A$ or neither is in $A$  
    \end{enumerate} 
 \end{enumerate}
\end{theorem}

Next, B. Shtylla, L. Traldi and L. Zulli in \cite{STZ09} presented an algebraic re-formulation of the Rosentiehl criteria of realizabity of Gauss diagrams as follows.

\begin{theorem}[The STZ-criteria of Realizability, {\cite[Theorem 2]{STZ09}}]~\\
 Let $\mathfrak{G}$ be Gauss diagram, $\Gamma(\mathfrak{G})$ its inrerlacement graph and $M(\mathfrak{G})$ its interlacement matrix. Then $\mathfrak{G}$ is realizable if and only if there exists a diagonal matrix $D$ such that $M+D$ is idempotent over the field $\mathsf{GF}(2)$.
\end{theorem}

\section{Counterexamples to some published descriptions} \label{sec:experimental}

It is claimed in \cite[Theorem 3.11]{GL} that 
 a Gauss diagram $\mathfrak{G}$ is realizable if and only if the following conditions hold:
  \begin{itemize}
    \item[(1)] the number of all chords that cross any two non-intersecting chords and the number of all chords intersecting each chord are both even (including zero),
    \item[(2)] for every chord $\mathfrak{c \in G}$ the Gauss diagram $\widehat{\mathfrak{G}}_\mathfrak{c}$ (that is, Conway's smoothing of the chord $c$) also satisfies the above condition.
  \end{itemize}

 It was also shown (see the proof of Theorem 4.3 in \cite{GL}) that these two conditions can be reformulated in term of adjacency matrix $M$ of the interlacement graph of $\mathfrak{G}$ as follows (the first two conditions are the evenness conditions, and together the 3 conditions form a partial case of the STZ-conditions):
   \begin{enumerate}
      \item $\langle m_i, m_j \rangle \equiv 0(\bmod{2})$, $1 \le i \le n$,
      \item $\langle m_i, m_j \rangle \equiv 0 (\bmod{2})$, if the corresponding chords do not intersect,
      \item $\langle m_i, m_j \rangle + \langle m_i, m_k \rangle + \langle m_j, m_k \rangle \equiv 1 (\bmod{2})$, if the corresponding chords intersect pairwise.
  \end{enumerate}

Another publication \cite{B19}, which quotes an earlier version of \cite{GL}, repeats the same claim, namely, that a Gauss diagram is realizable if and only if its interlacement graph satisfies these three conditions.


We have implemented these conditions in a computer tool \cite{KLV21-lintel} and conducted  extensive computational experiments for verification of these (and other) conditions for realiziability and for enumeration of the classes of Gauss diagrams, see   \cite{KLV21-experimental, KLV21-circle}. We found a family of counterexamples showing that these conditions are necessary but not sufficient for a Gauss diagram to be realizable. See also \cite{KLV22} where we apply deep learning to understand why these counterexamples remained unnoticed for some time.


The three conditions above describe realizability correctly for Gauss diagrams with up to eight chords. There is exactly one Gauss diagram with nine chords, see Fig.\ref{counterexample_to_n=9}, which satisfies the above conditions but is not realizable. There are $6$ Gauss diagrams with ten chords which satisfy the above conditions but are not realizable, see Fig.\ref{fig:counter10}. The results for Gauss diagrams of various sizes are summarised in Table \ref{tab:table1}; it shows how many Gauss diagrams satisfy the STZ-conditions and the above three conditions.

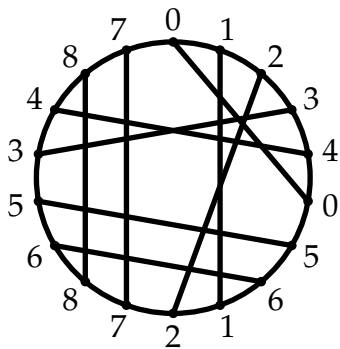
\begin{figure}[h!]
    \centering
     \begin{tikzpicture}[scale=0.6]
        \draw[line width =2] (0,0) circle (3);
     
     {\foreach \angle/ \label in
   { 90/0, 110/7, 130/8, 150/4, 170/3, 190/5, 210/6, 230/8, 250/7, 
     270/2, 290/1, 310/6, 330/5, 350/0, 10/4, 30/3, 50/2, 70/1 }
   {
    \fill(\angle:3.5) node{$\label$};
    \fill(\angle:3) circle (3pt) ;
    }
}

  \draw[line width = 2] (90:3) -- (350:3);
  \draw[line width = 2] (70:3) -- (290:3);
  \draw[line width = 2] (50:3) -- (270:3);
  \draw[line width = 2] (30:3) -- (170:3);
  \draw[line width = 2] (10:3) -- (150:3);
  \draw[line width = 2] (330:3) -- (190:3);
  \draw[line width = 2] (310:3) -- (210:3);
  \draw[line width = 2] (250:3) -- (110:3);
  \draw[line width = 2] (230:3) -- (130:3);

  \end{tikzpicture}
    \caption{This is only one Gauss diagram with nine chords which satisfies above conditions (1)--(3) but is not realizable.}
    \label{counterexample_to_n=9}
\end{figure}

\begin{figure}[h!]
 \begin{tikzpicture}[scale=0.8]
   \begin{scope}[line width = 2, scale=0.7]
\draw (0,0) circle (3);

 \draw[line width = 2] (90:3) -- (324:3);
 \draw[line width = 2] (162:3) -- (360:3);
 
 \draw[line width = 2] (18:3) -- (72:3);
 \draw[line width = 2] (126:3) -- (216:3);
 \draw[line width = 2] (108:3) -- (234:3);
 \draw[line width = 2] (198:3) -- (288:3);
 
 \draw[line width = 2] (144:3) -- (342:3);
 \draw[line width = 2] (180:3) -- (306:3);
 
 \draw[line width = 2] (252:3) -- (54:3);
 \draw[line width = 2] (36:3) -- (270:3);


  \end{scope}
  \begin{scope}[xshift = 6 cm,scale=0.7, line width = 2]
    \draw[line width =2] (0,0) circle (3);

 \draw[line width = 2] (162:3) -- (360:3);
 \draw[line width = 2] (90:3) -- (180:3);
 
 \draw (108:3) -- (234:3);
 \draw (126:3) -- (253:3);
 \draw (216:3) -- (306:3);
 \draw (216:3) -- (306:3);
 \draw (72:3) -- (18:3);
 
 \draw[line width = 2] (288:3) -- (54:3);
 \draw[line width = 2] (36:3) -- (270:3);

 \draw[line width = 2] (144:3) -- (342:3);
 \draw[line width = 2] (324:3) -- (198:3);

  \end{scope}
 
\begin{scope}[xshift = 12cm,scale = 0.7,line width = 2]
    \draw[line width =2] (0,0) circle (3);

 \draw[line width = 2] (324:3) -- (234:3);
 \draw[line width = 2] (72:3) -- (18:3);
 
 \draw (108:3) -- (198:3);
 \draw (162:3) -- (288:3);
 \draw (180:3) -- (306:3);
 \draw (36:3) -- (270:3);
 \draw (72:3) -- (18:3);
 
 \draw[line width = 2] (360:3) -- (126:3);
 \draw[line width = 2] (144:3) -- (342:3);
 
 \draw[line width = 2] (90:3) -- (216:3);
 \draw[line width = 2] (252:3) -- (54:3);

\end{scope}
\begin{scope}[yshift = -6 cm,scale = 0.7,line width = 2]
    \draw[line width =2] (0,0) circle (3);

 \draw[line width = 2] (252:3) -- (162:3);
 \draw[line width = 2] (306:3) -- (180:3);
 
 \draw (90:3) -- (216:3);
 \draw (108:3) -- (234:3);
 \draw (126:3) -- (360:3);
 \draw (36:3) -- (270:3);
 \draw (72:3) -- (18:3);
 
 \draw[line width = 2] (36:3) -- (270:3);
 \draw[line width = 2] (54:3) -- (288:3);
 
 \draw[line width = 2] (144:3) -- (342:3);
 \draw[line width = 2] (198:3) -- (324:3);

  \end{scope}
   \begin{scope}[yshift= - 6cm,xshift = 6cm, scale = 0.7,line width =2]
    \draw[line width =2] (0,0) circle (3);

 \draw[line width = 2] (144:3) -- (18:3);
 \draw[line width = 2] (270:3) -- (360:3);
 \draw[line width = 2] (72:3) -- (342:3);
 \draw[line width = 2] (324:3) -- (198:3);

 \draw (108:3) -- (234:3);
 \draw (126:3) -- (216:3);

 \draw[line width = 2] (252:3) -- (54:3);
 \draw[line width = 2] (288:3) -- (90:3);
 
 \draw[line width = 2] (162:3) -- (36:3);
 \draw[line width = 2] (306:3) -- (180:3);


  \end{scope}
  \begin{scope}[yshift= - 6cm,xshift = 12cm, scale = 0.7,line width =2]
    \draw[line width = 2] (0,0) circle (3);

 \draw[line width = 2] (108:3) -- (234:3);
 \draw[line width = 2] (144:3) -- (54:3);

 \draw (306:3) -- (180:3);
 \draw (198:3) -- (288:3);
 \draw (360:3) -- (270:3);
 
 \draw[line width = 2] (72:3) -- (342:3);
 \draw[line width = 2] (324:3) -- (90:3);
 
 \draw[line width = 2] (126:3) -- (216:3);
 \draw[line width = 2] (162:3) -- (36:3);
\draw[line width = 2] (252:3) -- (18:3);

  \end{scope}
 \end{tikzpicture}
 \caption{All counterexamples for (1) -- (3) conditions for Gauss diagrams with ten chords.}
 \label{fig:counter10}
\end{figure}
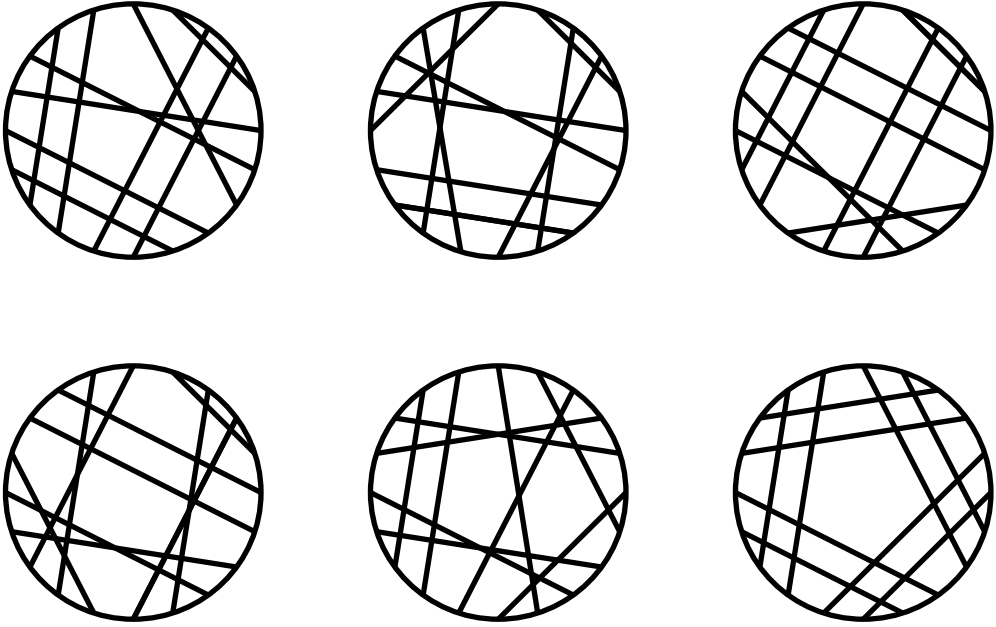

\begin{table}[h]
 \centering
    \begin{tabular}{|c|c|c|c|c|c|c|c|c|c|c|c|} 
   \hline 
   \it{Size of a Gauss diagram}  & 3 & 4 & 5 & 6 & 7 & 8 & 9 & 10 & 11 & 12
     \\
     \hline
      STZ & 1 & 1 & 2 & 3 & 10 & 27 & 101 & 364 & 1610 & 7202
      \\
      (1)--(3) & 1 & 1 & 2 & 3 & 10 & 27 & {\bf 102} & {\bf 370} & {\bf 1646} & {\bf 7437}
      \\
      \hline
    \end{tabular}
    
   \caption{The number of non equivalent Gauss diagrams of sizes  = 3, \ldots, 12, satisfying STZ-conditions and above conditions (1)--(3)}
\label{tab:table1}
\end{table}

\section{Realizability via linear equations in $\mathsf{GF}(2)$} \label{sec:STZ-reformulated}

We show that the realizability of a Gauss diagram is equivalent to the existence of a solution of the corresponding system of linear equations over the field $\mathsf{GF}(2)$.

Let $M$ be the adjacency matrix of a Gauss diagram $\mathfrak{G}$. Being $M$ symmetric we then have $M^2 = (\langle m_i, m_j \rangle)_{1\le i,j \le n}$, over $\mathsf{GF}(2)$, where 
\[
\langle m_i, m_j\rangle: = m_{i,1}m_{j,1} + \cdots + m_{i,n}m_{j,n},
\]
and $m_k:=(m_{k,1}, \ldots, m_{k,n})$ is the $k$th row of the $M$.

Let $D$ be a diagonal $n\times n$ matrix, \textit{i.e.,} $D = \sum_{k \in K} E_{k,k}$, where $K \subseteq \{1,\ldots, n\}$, and $E_{k,k}$ the elementary matrix, that is $E_{k,k} = (e_{i,j})_{1\le i,j \le n}$ where $e_{i,j} = 1$ if and only if $i=j=k$ and $e_{i,j}=0$ otherwise. Since $M$ is symmetric then, if $K \ne \{1,\ldots,n\}$, $DM + MD = \sum_{k \in K} M_k$ where $M_k$ is obtained from $M$ by zeroing all elements of $M$ except elements of the $k$th row and the $k$th column.

Note that if $K = \{1,\ldots, n\}$ then $D$ is the identity $n\times n$ matrix, and the STZ-criteria implies $M^2 = M$, i.e, $M$ is idempotent.

Thus the STZ-criteria can be reformulated as follows

\begin{proposition}
 Let $\mathfrak{G}$ be a Gauss diagram and $M$ the adjacency $n \times n$ matrix of its interlacement graph $\Gamma(\mathfrak{G})$. Then $\mathfrak{G}$ is realizable if and only if one of the following conditions hold:
 \begin{enumerate}
     \item whenever $m_{i,j} \equiv \mathbf{X} (\bmod{2})$ then $\langle m_i, m_j \rangle \equiv \mathbf{X} (\bmod{2})$, $1\le i,j \le n$, $\mathbf{X} = 0,1,$
     \item there is a $K\subsetneq \{ 1,\ldots, n \}$ such that $M + M^2 = \sum_{k\in K}M_k.$
 \end{enumerate}  
\end{proposition}
\begin{proof}
 Indeed, the first condition holds when the matrix $M$ is idempotent, $M^2 = M$ and by the STZ-criteria, $D$ is zero matrix. Next, if $M$ is not idempotent the STZ conditions imply that a diagonal matrix $D$ mast have a form $D = \sum_{k \in K}E_k$ where $K \subsetneq \{1,\ldots, n\}$ and the statement follows.
\end{proof}
 
 We thus can reformulate the STZ-criteria as follows
 
 \begin{proposition}\label{reformSTZ}
  Let $\mathfrak{G}$ be a Gauss diagram, $M = (m_{i,j})_{1\le i,j \le n}$ its adjacency matrix. Then the diagram $\mathfrak{G}$ is realizable if and only if the following system of equations
    \[
   \left\{  m_{i,j} \mathbf{X}_i + m_{i,j}\mathbf{X}_j = \langle m_i, m_j \rangle + m_{i,j}, \quad 1\le i,j \le n  \right.
  \]
 has a solution over a field $\mathsf{GF}(2).$ 
 \end{proposition}
 \begin{proof}
  Let $M$ be an adjacency matrix of a Gauss diagram $\mathfrak{G}$, say $M = (m_{i,j})_{1\le i,j \le n}$. Set $M':=M + M^2$, $M'=(m'_{i,j})_{1 \le i,j \le n}$. We then obtain $m'_{i,j} = m_{i,j} + \langle m_i, m_j \rangle$, for all $1 \le i,j \le n.$ 

Next, since $M$ is symmetric with zero diagonal we then get
 \[
  \sum_{i=1}^n \mathbf{X}_i M_i = \begin{pmatrix}
   0 & (\mathbf{X}_1 + \mathbf{X}_2)m_{1,2} & (\mathbf{X}_1 + \mathbf{X}_3)m_{1,3} & \cdots & (\mathbf{X}_1 + \mathbf{X}_n)m_{1,n}\\
   (\mathbf{X}_1 + \mathbf{X}_2)m_{1,2} & 0 & (\mathbf{X}_2 + \mathbf{X}_3)m_{2,3} & \cdots & (\mathbf{X}_2 + \mathbf{X}_n)m_{2,n} \\
   \vdots & \vdots & \vdots & \ddots & \vdots & \\
   (\mathbf{X}_1 + \mathbf{X}_n)m_{1,n} & (\mathbf{X}_2 + \mathbf{X}_n)m_{2,n} & (\mathbf{X}_3 + \mathbf{X}_n)m_{3,n} & \cdots & 0
  \end{pmatrix}
 \]
 where $\mathbf{X}_1,\ldots, \mathbf{X}_n \in \mathsf{GF}(2).$
 
 Thus, by the STZ-criteria, $M = M^2 + \sum_{k =1}^n \mathbf{X}_k M_k$. Hence we get the following system of equations
 \[
 \Bigl\{ m_{i,j}\mathbf{X}_i + m_{i,j}\mathbf{X}_j = m'_{i,j}, \qquad 1 \le i,j \le n, \Bigr.
 \]
and the statement follows. It is clear that in the case $K = \{1,\ldots, n\}$ the statement holds because STZ-criteria implies that $M^2 = M$, \textit{i.e.,} all $\mathbf{X}_i = 1$. 
 \end{proof}

\begin{remark}
Let us show that STZ-criteria implies the condition (2) of Definition \ref{evennes}.
  Indeed, let $M = (m_{i,j})_{1 \le i,j \le n}$ be its adjacency matrix. By the assumptions there exist at least two $i,j$ such that $m_{i,j} = 0$ and $\langle m_i, m_j \rangle = 1$. Hence the system contains the equations $0=1$ that gives a contradiction. 
\end{remark}

Thus in practice to know whether a Gauss diagram is realizable it is useful firstly to check the evenness conditions are holding. 

\begin{corollary}\label{cor_of_system}
 Let a Gauss diagram $\mathfrak{G}$ satisfy the evenness conditions, $M = (m_{i,j})_{1 \le i,j \le n}$ its adjacency matrix. Let $K \subseteq \{1,\ldots, n\}\times \{1,\ldots, n\}$ be a subset such that whenever $(i,j) \in K$ then $m_{i,j} = 1$. Then $\mathfrak{G}$ is realizable if and only if the following system of equations
 \[
  \mathbf{X}_i + \mathbf{X}_j = \langle m_i, m_j \rangle + 1, \qquad (i,j) \in K
 \]
 has a solution over the field $\mathsf{GF}(2)$. 
\end{corollary}

\begin{proof}
 Indeed, if $\mathfrak{G}$ satisfies the evenness conditions then whenever $m_{i,j} =0$ we have $\langle m_i, m_j \rangle = 0$ and by Proposition \ref{reformSTZ} the statement follows.
\end{proof}

\begin{example}\label{razd.trilistnik}
  Let us show that the Gauss diagram in Fig.\ref{fig_for_STZ1_not_realizable} is not realizable.
  
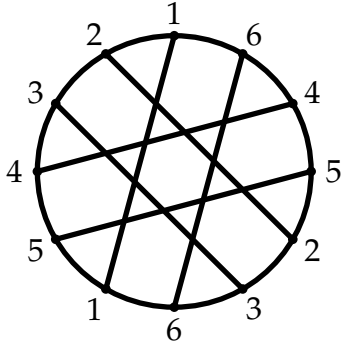
\begin{figure}[h!]
     \begin{tikzpicture}[scale=0.6]
      \draw[line width =2] (0,0) circle (3);
      {\foreach \angle/ \label in
       { 90/1, 120/2, 150/3, 180/4, 210/5, 240/1, 270/6, 300/3, 330/2, 
        0/5, 30/4, 60/6 
        }
     {
        \fill(\angle:3.5) node{$\label$};
        \fill(\angle:3) circle (3pt) ;
      }
    }
  \draw[line width = 2] (90:3) -- (240:3);
  \draw[line width = 2] (120:3) -- (330:3);
  \draw[line width = 2] (150:3) -- (300:3);
  \draw[line width = 2] (180:3) -- (30:3);
  \draw[line width = 2] (210:3) -- (0:3);
  \draw[line width = 2] (270:3) -- (60:3);
     \end{tikzpicture}
\end{figure}

Its adjacency matrix is

\[
 M = \bordermatrix{
    & 1 & 2 & 3 & 4 & 5 & 6 \cr
   1 & 0 & 0 & 1 & 1& 1& 1 \cr
   2 & 0 & 0 & 1 & 1& 1& 0 \cr
   3 & 1 & 1 & 0 & 1 & 1& 1 \cr
   4& 1 & 1 & 0 & 0 & 1& 1 \cr
   5& 1 & 1 & 1 & 1 & 0& 0 \cr
    6& 1 & 1 & 1 & 1 & 0& 0 \cr
 }
\]

We get $K=\{(1,3), (1,4), (1,5), (1,6), (2,3), (2,4), (2,5), (2,6), (3,5), (3,6), (4,5), (4,6)\}$. Hence, the corresponding system of equations has the following equations
\begin{align*}
    & \mathbf{X}_1 + \mathbf{X}_3 = 1,\\
    & \mathbf{X}_1 + \mathbf{X}_5 = 1, \\
    & \mathbf{X}_3 + \mathbf{X}_5 = 1,
\end{align*}
which show that the system has no a solution, thus the diagram is not realizable.
\begin{flushright}
 $\square$
\end{flushright}
\end{example}

\begin{example}\label{STZ_example}
  Let us consider the following Gauss diagram and show that it is realizable. 
 
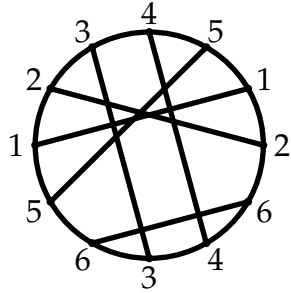
\begin{figure}[h!]
     \centering
     \begin{tikzpicture}[scale = 0.5]
      \draw[line width =2] (0,0) circle (3);
      {\foreach \angle/ \label in
       { 90/4, 120/3, 150/2, 180/1, 210/5, 240/6, 270/3, 300/4, 330/6, 
        0/2, 30/1, 60/5 
        }
     {
        \fill(\angle:3.5) node{$\label$};
        \fill(\angle:3) circle (3pt) ;
      }
    }
  \draw[line width = 2] (90:3) -- (300:3);
  \draw[line width = 2] (120:3) -- (270:3);
  \draw[line width = 2] (150:3) -- (0:3);
  \draw[line width = 2] (180:3) -- (30:3);
  \draw[line width = 2] (210:3) -- (60:3);
  \draw[line width = 2] (240:3) -- (330:3);
     \end{tikzpicture}
 \end{figure}
 
 We have
 \[
 M = \bordermatrix{
    & 1 & 2 & 3 & 4 & 5 & 6 \cr
   1 & 0 & 1 & 1 & 1& 1& 0 \cr
   2 & 1 & 0 & 1 & 1& 1& 0 \cr
   3 & 1 & 1 & 0 & 0 & 1& 1 \cr
   4& 1 & 1 & 0 & 0 & 1& 1 \cr
   5& 1 & 1 & 1 & 1 & 0& 0 \cr
    6& 0 & 0 & 1 & 1 & 0& 0 \cr
 }, \qquad
  M^2 = \begin{pmatrix}
    0 & 1 & 0 & 0 & 1 & 0 \\
    1 & 0 & 0 & 0 & 1 & 0 \\
    0 & 0 & 0 & 0 & 0 & 0 \\
    0 & 0 & 0 & 0 & 0 & 0 \\
    1 & 1 & 0 & 0 & 0 & 0 \\
    0 & 0 & 0 & 0 & 0 &  0
 \end{pmatrix}
\]

We thus have the following system of equations
 \[
  \begin{cases}
   \mathbf{X}_1 + \mathbf{X}_2  = 0 \\
   \mathbf{X}_1 + \mathbf{X}_3  = 1 \\
   \mathbf{X}_1 + \mathbf{X}_4  = 1 \\
   \mathbf{X}_1 + \mathbf{X}_5  = 0 \\
   \mathbf{X}_2 + \mathbf{X}_3  = 1 \\
   \mathbf{X}_2 + \mathbf{X}_4  = 1 \\
   \mathbf{X}_2 + \mathbf{X}_5  = 0 \\
   \mathbf{X}_3 + \mathbf{X}_5  = 1 \\
   \mathbf{X}_3 + \mathbf{X}_6  = 1 \\
   \mathbf{X}_4 + \mathbf{X}_5  = 1 \\
   \mathbf{X}_4 + \mathbf{X}_6  = 1
  \end{cases}
 \]

It follows that $\mathbf{X}_1 = c, \mathbf{X}_2 = c, \mathbf{X}_3 = 1+c, \mathbf{X}_4 = 1+ c, \mathbf{X}_5 = c, \mathbf{X}_6 = c$, where $c \in \mathsf{GF}(2)$ \textit{i.e.,} the system has a solution and hence the diagram is realizable.
\begin{flushright}
 $\square$
\end{flushright}
\end{example}

 For a given Gauss diagram $\mathfrak{G}$ and its interlacement graph $\Gamma(\mathfrak{G}) = (V,E)$, where $V$ is a set of vertices and $E$ is a set of edges, we consider a weighted graph ${\Gamma_\omega(\mathfrak{G}}):=(V,E,\omega)$, where the edge weight function $\omega: E \to \mathsf{GF}(2)$ is defined as follows 
 \[
 \omega((i,j)) := \langle m_i, m_j \rangle,
\] 
where $m_i, m_j$ are the $i$th and the $j$th rows of the adjacency matrix $M$ of $\Gamma(\mathfrak{G})$.

\begin{corollary}\label{cor_for_graph}
A Gauss diagram $\mathfrak{G}$ is realizable if and only if its interlacemnt graph $\Gamma(\mathfrak{G})$ is euler and for any cycle $C = (c_1,\ldots, c_\ell)$ of the $\Gamma(\mathfrak{G})$, we have 
 \[
 \sum_{i=1}^\ell \omega(c_i) \equiv \ell \bmod{2}.
 \]
\end{corollary}
\begin{proof}
 Indeed, by Corollary \ref{cor_of_system}, 
 \[
  \begin{cases}
   \mathbf{X}_{i_1} + \mathbf{X}_{i_2} \equiv \omega((i_1,i_2)) + 1, \\
   \phantom{\mathbf{X}_{i_1}\,}\vdots \phantom{+\mathbf{X}_{i_2}} \ddots \phantom{\omega((i_1,i_2))} \vdots \\
   \mathbf{X}_{i_\ell} + \mathbf{X}_{i_1} \equiv \omega((i_\ell, i_1)) + 1,
  \end{cases} 
 \]
 where we have put $c_1 = (i_1,i_2), \ldots, c_\ell = (i_\ell, i_1)$, and the statement thus follows.
\end{proof}

Let us turn the previous examples and apply the previous criteria.

\begin{example}
Let us consider the Gauss diagram $\mathfrak{G}$ from Example \ref{razd.trilistnik} (see Fig.\ref{fig_for_STZ1_not_realizable}). It is easy to see that $M^2 = 0$ for its adjacency matrix $M$. Next, consider its weighted graph ${\Gamma_\omega(\mathfrak{G})}$; all its edges have thus weight equal to $0$. We see that this graph does not satisfy the conditions of Corollary \ref{cor_for_graph}. Indeed, cycle $(1,2,5)$ has length $3$ but $\omega((1,2)) + \omega((2,5)) + \omega((5,1)) \equiv 0 (\bmod{2})$. Hence the diagram is not realizable.

 \begin{figure}[h!]
     \centering
     \begin{tikzpicture}[scale = 0.5]
      \draw[line width =2] (0,0) circle (3);
      {\foreach \angle/ \label in
       { 90/1, 120/2, 150/3, 180/4, 210/5, 240/1, 270/6, 300/3, 330/2, 
        0/5, 30/4, 60/6 
        }
     {
        \fill(\angle:3.5) node{$\label$};
        \fill(\angle:3) circle (3pt) ;
      }
    }
  \draw[line width = 2] (90:3) -- (240:3);
  \draw[line width = 2] (120:3) -- (330:3);
  \draw[line width = 2] (150:3) -- (300:3);
  \draw[line width = 2] (180:3) -- (30:3);
  \draw[line width = 2] (210:3) -- (0:3);
  \draw[line width = 2] (270:3) -- (60:3);
  
  \begin{scope}[xshift = 10cm]
     {\foreach \angle/ \label in
       { 90/1, 150/6,  210/5, 270/4,  330/3, 30/2 
        }
     {
        \fill(\angle:3.5) node{$\label$};
        \fill(\angle:3) circle (5pt) ;
      }
    }
     
  \draw (90:3) -- (30:3);     
  \draw (90:3) -- (330:3);
  \draw (90:3) -- (270:3);
  \draw (90:3) -- (210:3);

  \draw (30:3) -- (270:3);
  \draw (30:3) -- (210:3);
  \draw (30:3) -- (150:3);
  
  \draw (330:3) -- (270:3);
  \draw (330:3) -- (210:3);
  \draw (330:3) -- (150:3);
  
  \draw (270:3) -- (150:3);
  
  \draw (210:3) -- (150:3);
  \end{scope}
     \end{tikzpicture}
     \caption{}\label{fig_for_STZ1_not_realizable}
 \end{figure}
\end{example}

\begin{example}
Let us consider the Gauss diagram from Example \ref{STZ_example} (see Fig.\ref{fig_for_STZ1}). Knowing its adjacency matrix $M$ and $M^2$ we obtain
\[
 \omega(e)  = \begin{cases}
  1, & \mbox{if $e \in \{(1,2), (1,5), (2,5)\}$},\\
  0, & \mbox{in otherwise.}
 \end{cases}
\]

By the straightforward verification we see that the graph $\Gamma(\mathfrak{G})$ satisfies the conditions of Corollary \ref{cor_for_graph} and thus the diagram is realizable.
 \begin{figure}[h!]
     \centering
     \begin{tikzpicture}[scale = 0.5]
      \draw[line width =2] (0,0) circle (3);
      {\foreach \angle/ \label in
       { 90/4, 120/3, 150/2, 180/1, 210/5, 240/6, 270/3, 300/4, 330/6, 
        0/2, 30/1, 60/5 
        }
     {
        \fill(\angle:3.5) node{$\label$};
        \fill(\angle:3) circle (3pt) ;
      }
    }
  \draw[line width = 2] (90:3) -- (300:3);
  \draw[line width = 2] (120:3) -- (270:3);
  \draw[line width = 2] (150:3) -- (0:3);
  \draw[line width = 2] (180:3) -- (30:3);
  \draw[line width = 2] (210:3) -- (60:3);
  \draw[line width = 2] (240:3) -- (330:3);
  
  \begin{scope}[xshift = 10cm]
     {\foreach \angle/ \label in
       { 90/1, 150/6,  210/5, 270/4,  330/3, 30/2 
        }
     {
        \fill(\angle:3.5) node{$\label$};
        \fill(\angle:3) circle (5pt) ;
      }
    }
     
  \draw[line width =2] (90:3) -- (30:3);     
  \draw (90:3) -- (210:3);
  \draw (90:3) -- (330:3);
  \draw (90:3) -- (270:3);
  \draw[line width =2] (90:3) -- (210:3);
  
  \draw (30:3) -- (330:3);
  \draw (30:3) -- (270:3);
  \draw[line width =2] (30:3) -- (210:3);
  
  \draw (330:3) -- (210:3);
  \draw (330:3) -- (150:3);
  
  \draw (270:3) -- (210:3);
  \draw (270:3) -- (150:3);
    \end{scope}
     \end{tikzpicture}
     \caption{The Gauss diagram $\mathfrak{G}$ and its weighted graph $\Gamma_\omega(\mathfrak{G})$; the thick edges have weight equal to $1$ and the other have weight equal to 0. We see that this graph satisfies the conditions of Corollary \ref{cor_for_graph} and thus the diagram is realizable}
     \label{fig_for_STZ1}
 \end{figure}
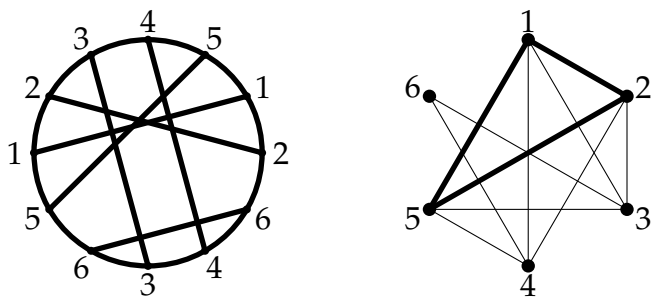
\end{example}

\section{Touch realization of Gauss diagrams and Dehn's Criteria}

In this section we consider the problem of touch realization of Gauss diagrams. The results in this section are not, strictly speaking, new, because they form a part of some descriptions of realizable Gauss diagrams; however, they are never presented on their own and explicitly. We essentially follow the definitions and concepts of \cite{dFOdM97}. The results in this section are also very similar to \cite[Lemma 3]{RT84}. 

A \textit{parametrized curve} $\zeta$ is a continuous mapping $\zeta:[0,1] \to \mathbb{R}^2$ such that $\zeta(0) = \zeta(1)$ and for which the underlying curve $\zeta([0,1])$ is piecewise smooth and has a finite number of multiple points, all of which have multiplicity two. Let $\mathrm{P}(\zeta)$ denote the set of the points of multiplicity two. To any point $p\in \mathrm{P}(\zeta)$, we associate the two parameter values $t_p',t_p''\in [0,1]$ such that $t_p'< t_p''$ and $\zeta(t_p') = \zeta(t_p'') = p$. A point $p\in \mathrm{P}(\zeta)$ is a \textit{touching point} if any local deformation of $\zeta$ in a neighborhood of $t_p'$ does not preserve the existence of a double point. A \textit{touch curve} is a parametrized curve with only touching points.

For a given touch curve $\zeta$ we can similarly associate a Gauss code; walking on a path along the $\zeta$ until returning back to the origin and then generate a double occurrence word of the curve $\zeta$. Putting the labels of the crossings on a circle in the order of the word and joining by a chord all pairs of identical labels we then obtain a chord diagram (=Gauss diagram) of a touch curve $\mathfrak{G}(\zeta)$ (see Fig.\ref{Touch_Gauss})

 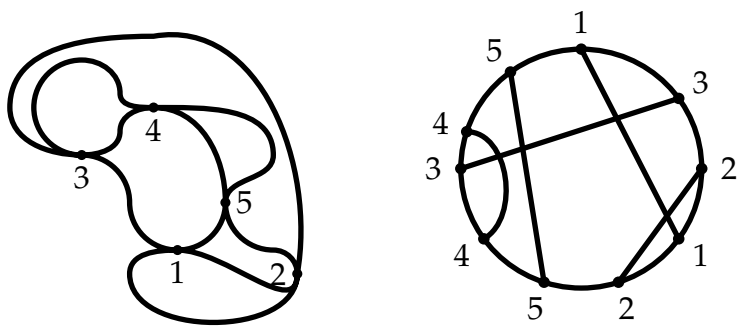
\begin{figure}[h!]
    \centering
    \begin{tikzpicture}[scale=0.9]
    \begin{scope}[scale=0.7]
     \draw[line width =2] (1, 1) to [out = 0, in = 90] (2,0) to [out = 270, in = 180] (3,-1) to [out = 0, in = 270] (4,0) to [out = 90, in = 0] (2.5, 2) to [out = 180, in = 90] (1.8,1.5) to [out = 270, in = 0] (1,1) to [out = 180, in = 270] (0,2) to [out = 90, in = 180] (1,3) to [out = 0, in = 90] (1.8, 2.3) to [out = 270, in = 180] (2.5,2) to [out = 0, in = 90] (5, 1) to  [out = 270, in =90] (4,0) to [out = 270, in = 180] (5,-1) to [out = 0, in= 90] (5.5,-1.5) to [out=270, in = 0] (3,-1) to [out = 180, in = 90] (2,-1.5) to [out = 270, in = 270] (5.5, -1.5) to [out = 80, in = 10] (2.5,3.5) to [out = 180, in = 90] (-0.5,2) [out = 270, in = 180] to (1,1);

  \fill(1,1) node[below]{$3$};
  \fill (1,1) circle (3pt);

  \fill(3,-1) node[below]{$1$};
  \fill (3,-1) circle (3pt);

  \fill(4,0) node[right]{$5$};
  \fill (4,0) circle (3pt);

  \fill(2.5,2) node[below]{$4$};
  \fill (2.5,2) circle (3pt);

  \fill(5.5,-1.5) node[left]{$2$};
  \fill(5.5,-1.5) circle (3pt);      
    \end{scope}


  \begin{scope}[xshift = 8cm, yshift = 0.5cm, scale =0.8,line width = 2]
     \draw[line width =2] (0,0) circle (2.2);
     
     {\foreach \angle/ \label in
   { 90/1, 126/5, 162/4,  180/3, 216/4, 252/5, 288/2, 324/1, 0/2, 36/3}
   {
    \fill(\angle:2.7) node{$\label$};
    \fill(\angle:2.2) circle (3pt) ;
    }
}
     \draw (90:2.2) -- (324:2.2);
     \draw (126:2.2) -- (252:2.2);
     \draw (162:2.2) to [out = 0, in = 30] (216:2.2);
     \draw (180:2.2) -- (36:2.2);
     \draw (288:2.2) -- (0:2.2);
\end{scope}
\end{tikzpicture}
\caption{A touch curve and its Gauss diagram}
    \label{touch_curve_and_diagram}
\end{figure}

Touch realization of Gauss diagrams has been employed in some steps of some algorithms for checking realizability of Gauss diagrams, including the first description of realizability, in which M. Dehn calls them ``Baum--Zwiebel Figur'' (``tree-and-onion diagram'') \cite{D} (see Fig.\ref{touch_and_graph}) and \cite{RT84}.

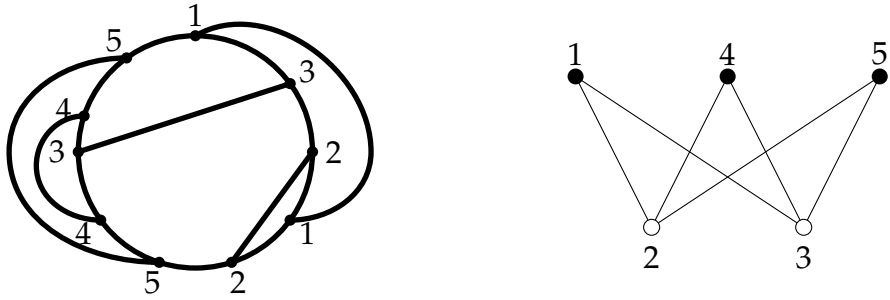
\begin{figure}
    \centering
     \begin{tikzpicture}
    \begin{scope}[xshift = 0, yshift = 0cm, scale =0.7,line width=2]
     \draw[line width =2] (0,0) circle (2.2);
     
     {\foreach \angle/ \label in
   { 90/1, 126/5, 162/4,  180/3, 216/4, 252/5, 288/2, 324/1, 0/2, 36/3}
   {
    \fill(\angle:2.6) node{$\label$};
    \fill(\angle:2.2) circle (3pt) ;
    }
}
     \draw (90:2.2) to [out = 30, in = 90] (0:3.3) to [out=270, in = 0] (324:2.2);
     \draw (126:2.2) to [out = 180, in = 90] (180:3.5) to [out = 270, in = 180] (252:2.2);
     \draw (162:2.2) to [out = 180, in = 90] (185: 3) to [out = 270, in =180]  (216:2.2);
     \draw (180:2.2) -- (36:2.2);
     \draw (288:2.2) -- (0:2.2);
\end{scope}
\begin{scope}[xshift = 5cm]
 
    \draw (0,1) -- (1,-1);
    \draw (0,1) -- (3,-1);

    \draw (2,1) -- (1,-1);
    \draw (2,1) -- (3,-1);

    \draw (4,1) -- (1,-1);
    \draw (4,1) -- (3,-1);
   
    \fill (0,1.3) node {$1$};
    \fill  (0,1) circle (3pt);

    \fill (2,1.3) node {$4$};
    \fill  (2,1) circle (3pt);

    \fill (4,1.3) node {$5$};
    \fill  (4,1) circle (3pt);

    \fill[white] (1,-1) circle (3pt);

    \fill (1,-1.1) node[below] {$2$};
    \draw  (1,-1) circle (3pt);

   \fill[white] (3,-1) circle (3pt);

   \fill (3,-1.1) node[below] {$3$};
    \draw  (3,-1) circle (3pt);
\end{scope}
     \end{tikzpicture}
    \caption{The previous touch realizable Gauss diagram can be also shown as follows; we see that a possibility to laying out some chords inside and outside exactly correspondences to its interlacement graph are to be bipartite.}
    \label{touch_and_graph}
\end{figure}

\begin{theorem}{\cite[\S 3]{D}}
  A Gauss code is realizable if and only if it satisfies both the following conditions:
  \begin{enumerate}
      \item (Gauss' parity condition) any matching pair of symbols must be separated by an even number of other symbols,
      \item (Dehn's untangling condition) after reversing every substring bounded by matching symbols the corresponding interlacement graph must be bipartite. 
  \end{enumerate}
\end{theorem}

It is clear that Gauss' parity condition is a partial case of the evenness condition (see Definition \ref{evennes}), that is every chord is crossed by an even (possibly zero) number of chords.

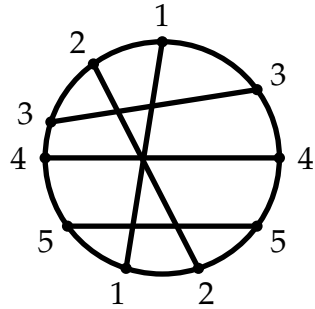
\begin{figure}
    \centering
      \begin{tikzpicture}[scale = 0.7, line width=2]
    \draw[line width =2] (0,0) circle (2.2);
     
     {\foreach \angle/ \label in
   { 90/1, 126/2, 162/3,  180/4, 216/5, 252/1, 288/2, 324/5, 0/4, 36/3}
   {
    \fill(\angle:2.7) node{$\label$};
    \fill(\angle:2.2) circle (3pt) ;
    }
}
     
     \draw (90:2.2) -- (252:2.2);
     \draw (126:2.2) -- (288:2.2);
     \draw (162:2.2) -- (36:2.2);
     \draw (180:2.2) -- (0:2.2);
     \draw (324:2.2) -- (216:2.2);
      \end{tikzpicture}
    \caption{A Gauss diagram with a Gauss code $\mathbf{1234512543}$}\label{Gd+code}
    \label{Touch_Gauss}
\end{figure}

\begin{example}
Let us consider the following Gauss diagram (see Fig.\ref{Gd+code}). We have the following code $\mathbf{1234512543}$. It is clear that Gauss' parity condition holds. Reversing every substring bounded by matching symbols we obtain
\[
   \underbracket{\textbf{123451}}\textbf{2543} \to \textbf{1543}\underbracket{\textbf{212}}\textbf{543} \to \textbf{154}\underbracket{\textbf{3212543}} \to \textbf{15}{{434}}\textbf{52123}
\]
and
\[
\textbf{15}\underbracket{{434}}\textbf{52123} \to \textbf{1}\underbracket{\textbf{54345}}\textbf{2123} \to \textbf{1543452123}.
\]
 
The corresponding Gauss diagram and its interlacement graph is shown in Fig.\ref{G+C+Gr} We see that its interlacement graph is bipartite, hence the Gauss diagram in Fig.\ref{Gd+code} is realizable.

\begin{figure}[h!]
    \centering
      \begin{tikzpicture}
    \begin{scope}[xshift = 0cm, yshift = 0.5cm, scale =0.8,line width=2]
     \draw[line width =2] (0,0) circle (2.2);
     {\foreach \angle/ \label in
   { 90/1, 126/5, 162/4,  180/3, 216/4, 252/5, 288/2, 324/1, 0/2, 36/3}
   {
    \fill(\angle:2.5) node{$\label$};
    \fill(\angle:2.2) circle (3pt) ;
    }
}
     \draw (90:2.2) -- (324:2.2);
     \draw (126:2.2) -- (252:2.2);
     \draw (162:2.2) to [out = 0, in = 30] (216:2.2);
     \draw (180:2.2) -- (36:2.2);
     \draw (288:2.2) -- (0:2.2);
\end{scope}

\begin{scope}[xshift = 4cm, yshift = 0.5cm]
 
    \draw (0,1) -- (1,-1);
    \draw (0,1) -- (3,-1);

    \draw (2,1) -- (1,-1);
    \draw (2,1) -- (3,-1);

    \draw (4,1) -- (1,-1);
    \draw (4,1) -- (3,-1);
   
    \fill (0,1.3) node {$1$};
    \fill  (0,1) circle (3pt);

    \fill (2,1.3) node {$4$};
    \fill  (2,1) circle (3pt);

    \fill (4,1.3) node {$5$};
    \fill  (4,1) circle (3pt);

    \fill[white] (1,-1) circle (3pt);

    \fill (1,-1.1) node[below] {$2$};
    \draw  (1,-1) circle (3pt);

   \fill[white] (3,-1) circle (3pt);

   \fill (3,-1.1) node[below] {$3$};
    \draw  (3,-1) circle (3pt);
\end{scope}
      \end{tikzpicture}
    \caption{Gauss diagram and its interlacement graph which is bipartite.}
    \label{G+C+Gr}
\end{figure}
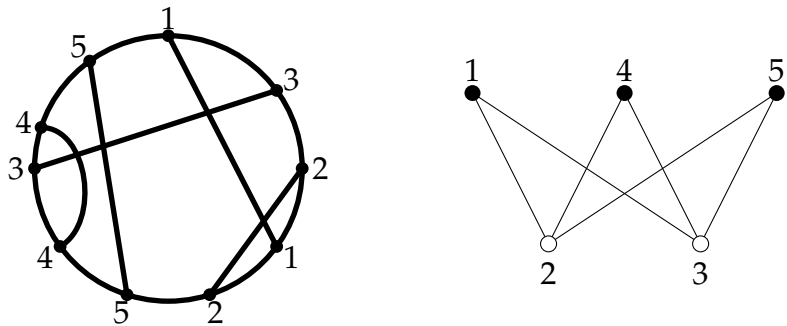
\end{example}

We see that in the criteria of realizability appears a bipartite condition for a graph. We are going to show that the problem of touch realization of a Gauss diagram (\textit{i.e.,} is there a touch curve $\zeta$ such that $\mathfrak{G} = \mathfrak{G}(\zeta)$?) can be solved in the same manner as the classical Gauss problem  by a system of equations (see Proposition \ref{reformSTZ}). To do so we prove the following lemma.

\begin{lemma}
  Let $\mathfrak{G}$ be a Gauss diagram, $M = (m_{i,j})_{1\le i,j \le n}$ its adjacency matrix. For each non-zero entry $m_{i,j}$, consider an equation $\mathbf{X}_i + \mathbf{X}_j = 1$. Thus we get a system of equation $S$. The diagram $\mathfrak{G}$ is touch-realizable if and only the system $S$ has a solution over a field $\mathsf{GF}(2).$ 
\end{lemma}
 
\begin{proof}
 Indeed, our aim is to lay out the chords inside and outside the circle to remove all intersections. We can conveniently encode this process by saying that $\mathbf{X}_i = 0$ (\textit{resp.} $\mathbf{X}_i = 1$) means that the $i$-th chord is drawn inside (\textit{resp.} outside) the circle. It is easy to see that solving the system $S$ is equivalent to laying out the chords in a way which removes all crossings. 
\end{proof}

\begin{proposition}
 \label{prop:planar-is-bipartite}
A Gauss diagram is touch-realizable if and only if its interlacement graph is bipartite. 
\end{proposition}
\begin{proof}
It is easy to see that a solution to the system of equations in the lemma exists if and only if the graph is bipartite. 
\end{proof}

\section{Realizability via bipartite graphs}

In this section, we aim to give a new description of the realizability of Gauss diagrams in terms of the adjacency matrix of its interlacement graphs by using STZ-criteria.

\begin{definition}\label{graph_for_Gauss}
 For a given graph $\Gamma(\mathfrak{G}) = (V,E)$ of a Gauss diagram $\mathfrak{G}$ we construct a new graph $\widetilde{\Gamma(\mathfrak{G}})$ as follows. For all pairs of vertices $v_i,v_j \in V$ such that $(v_i,v_j) \in E$ and the number of their common neighbors is odd we replace the edge $(v_i,v_j)$ by two new edges $(v_i, u_{i,j})$, and $(u_{i,j},v_j)$, where $u_{i,j}$ is a new vertex.
\end{definition}

\begin{theorem}\label{the_main}
  A Gauss diagram $\mathfrak{G}$ is realizable if and only if the evenness conditions hold and the graph $\widetilde{\Gamma(\mathfrak{G})}$ is bipartite.
\end{theorem}
\begin{proof}
 Let $M = (m_{i,j})_{1 \le i,j \le n}$ be an adjacency matrix of the $\mathfrak{G}$. By Corollary \ref{cor_of_system} the $\mathfrak{G}$ is realizable if and only if and only if the following system of equations
 \[
  \mathbf{X}_i + \mathbf{X}_j = \langle m_i, m_j \rangle + 1, \qquad (i,j) \in K
 \]
 has a solution over the field $\mathsf{GF}(2)$, where $K \subseteq \{1,\ldots, n\}\times \{1,\ldots, n\}$ is a subset such that whenever $(i,j) \in K$ then $m_{i,j} = 1$. 
 
 Next, let us decorate the graph $\Gamma(\mathfrak{G})$ as follows; mark any vertex $v_i$ by $\mathbf{X}_i$. Since $\mathbf{X}_i \in \mathsf{GF}(2)$ we can interpret this label as a color of the $v_i.$ It is clear that if $\langle m_i, m_j \rangle \equiv 1 (\bmod{2})$ then $\mathbf{X}_i = \mathbf{X}_j$,  and if $\langle m_i, m_j \rangle \equiv 0 (\bmod{2})$ then $\mathbf{X}_i \ne \mathbf{X}_j$. 
 
Further, it is clear that the conditions $m_{i,j}=1$, $\langle m_i, m_j \rangle \equiv 1 (\bmod{2})$ (\textit{resp.} $\langle m_i, m_j \rangle \equiv 0 (\bmod{2})$) can be reformulated in the terms of the graph $\Gamma(\mathfrak{G})$ as follows; $(v_i,v_j) \in E$ and the number of common neighbors for $v_i,v_j$ is odd (\textit{resp.} even).

Finally, if $(v_i,v_j) \in E$, $\langle m_i, m_j \rangle \equiv 1 (\bmod{2})$, and thus $\mathbf{X}_i = \mathbf{X}_j$, then by construction of $\widetilde{\Gamma(\mathfrak{G})}$ we have to add a new vertex $u_{i,j}$ and replace the edge $(v_i,v_j)$ by new two edges $(v_i, u_{i,j})$, $(u_{i,j},v_i)$ and hence any two vertexes with the same color will not adjacent to each other in new graph $\widetilde{\Gamma(\mathfrak{G})}$, \textit{i.e.,} we get a bipartite graph and the statement follows.
\end{proof}

\begin{remark}
Realiziability criteria proposed in Theorem~\ref{the_main} have simplest known to authors definition in terms of interlacement graphs in  a precise sense. Their formal definitions can be done via stating bipartite property for a graph $\widehat{\Gamma(\mathfrak{G}})$  \emph{definable} in terms of the interlacement graph using only \emph{first-order logic} extended by parity ($MOD_{2}$)  quantifier \cite{N00}. The previously known criteria such as in  \cite{LM,R76,RR,STZ09} require \emph{second-order} quantifiers to define them. Furthermore the criteria from  \cite{D}  also apply bipartite property checking but to a graph defined using an iterative procedure beyond  first-order logic extended by $MOD_{2}$. Further discussion of related definability questions can be found in \cite{KLV21-experimental}    
\end{remark}

\section*{Acknowledgments}
The work of the first and third named authors was supported by the Leverhulme Trust Research Project Grant RPG-2019-313.

\end{document}